\documentclass[a4paper,11pt,reqno]{amsart}

\usepackage[top=2.5cm, bottom=2cm, left=2.5cm, right=2.5cm, includefoot]{geometry}

\usepackage{amssymb,amsmath,amsthm,ascmac}
\usepackage[dvipdfmx]{graphicx}
\usepackage{color}

\usepackage[dvipdfm,colorlinks=true,bookmarks=true,citecolor=blue,
bookmarksnumbered=true,bookmarkstype=toc,linktocpage=true
]{hyperref}

\newtheorem{thm}{Theorem}[section]
\newtheorem{rem}[thm]{Remark}

\newtheorem{assum}[thm]{Assumption}

\newtheorem{ack}{}

\allowdisplaybreaks

\newcommand{\mcb}{\mathcal{B}}
\newcommand{\mcc}{\mathcal{C}}
\newcommand{\mcf}{\mathcal{F}}

\newcommand{\mbbr}{\mathbb{R}}
\newcommand{\mbbrp}{\mathbb{R}_{+}}

\newcommand{\mbX}{\mathbf{X}}

\newcommand{\al}{\alpha}

\newcommand{\sig}{\sigma}
\newcommand{\D}{\Delta}

\newcommand{\p}{\partial}

\def\nn{\nonumber}

\def\sumj{\sum_{j=1}^{n}}

\def\hatb{\hat{\beta}}

\title[Update estimation of diffusion parameter observed at high frequency]{
Update estimation of diffusion parameter\\ observed at high frequency
}

\author{Yusuke Shimizu}
\address{Graduate School of Mathematics, Kyushu University. 
744 Motooka, Nishi-ku, Fukuoka 819-0395, Japan.}
\email{y-shimizu@math.kyushu-u.ac.jp}

\date{\today}

\begin{document}
\maketitle

\begin{abstract}

We propose an update estimation method for a diffusion parameter from high-frequency dependent data under a nuisance drift element. We ensure the asymptotic equivalence of the estimator to the corresponding quasi-MLE, which has the asymptotic normality and the asymptotic efficiency. 
We give a simulation example to illustrate the theory.
\end{abstract}


\section{Introduction}
In general, under some appropriate conditions, $M$-estimators enjoy some nice properties, for example, the asymptotic efficiency of (quasi-)MLE and oracle properties of regularized estimators. 
However, it may suffer from heavy computation load since it is batch estimation in the sense that we have to optimize the appropriate objective function.
It would be of great help to be able to carry out recursive estimation, where we update the estimator by doing some fine tuning of the previous one. 
In this way, it enables us to process enormous data effectively. 
Recursive estimation is established as an application of stochastic approximation theory, which is mainly used in control system and computer science (see e.g., Robbins and Monro \cite{RobMon51}, Nevel'son and Khas'minskii \cite{NevHas73}, Dvoretzky \cite{Dvo86}, Borkar \cite{Bor08}).
The case of independent data was developed by, among others, Fabian \cite{Fab68} and \cite{Fab78}. 
We refer to Sharia \cite{Sha98}, \cite{Sha07}, \cite{Sha08} and \cite{Sha10_aism} for asymptotics of time-series model recursive estimator. 
See Kutoyants \cite[Section 2.6.6]{Kut04} and Lazrieva et al. \cite{LarShaTor08} as well as the references therein for the case of continuous-time data from a diffusion model.

\medskip

The model we consider through this paper is the scaled Wiener process with drift:
\begin{align}
X_{t}=\int_{0}^{t}\mu_{s}ds + \sqrt{\beta}w_{t},
\label{ex:wp-2}
\end{align}
where $\mu$ is a random and time-varying nuisance
process. Suppose that we observe only discrete-time sample $(X_{t^{n}_{j}})_{j=0}^{n}$ for $t^{n}_{j}=jh_{n}$ with $h_{n}\to 0$, where $T_{n}:=nh_{n}\rightarrow \infty$. 
We know that the quasi-MLE of a true value $\beta_{0}$ with regarding $\mu\equiv 0$ is
\begin{align}
\hatb_{n}^{qmle}=\frac{1}{T_{n}}\sumj(\D_{j}^{n}X)^{2},
\label{eq:qmle-beta}
\end{align}
where $\Delta_{j}^{n}X=X_{t_{j}^{n}}-X_{t_{j-1}^{n}}$, and $\hatb_{n}^{qmle}$ has the asymptotic normality and the asymptotic efficiency. The aim of this paper is to propose an update estimation method for the diffusion parameter $\beta$ under the nuisance drift element $\int \mu_{s}ds$. 
Specifically, 
we wish to construct a sequence of estimates $\hat{\beta}_{0}^{n},\hat{\beta}_{1}^{n},\hat{\beta}_{2}^{n},\ldots,\hat{\beta}_{n}^{n}$, computed in a recursive manner, in such a way that for any $j\leq n$, the difference $\hat{\beta}_{j}^{n}-\hat{\beta}_{j-1}^{n}$ is a function described by $\hat{\beta}_{j-1}^{n}$, $h_{n}$, $j$, and some (not all) data, 
and $\hat{\beta}_{n}^{n}$ exhibits suitable convergence property for an arbitrary initial value $\hat{\beta}_{0}^{0}$.
Usually this estimator does not require any numerically hard optimization, while its asymptotic behavior does require careful investigation. 
Note the difference from the recursive (online) estimation for sample $X_{1},X_{2},\dots$, where each successive estimator is obtained by using data observed one after another. 
We will bring about good asymptotic property of $\hat{\beta}_{n}^{n}$ as efficient as the quasi-MLE, i.e we ensure the asymptotic equivalence of $\hat{\beta}_{n}^{n}$ to $\hat{\beta}_{n}^{qmle}$:
\begin{align}
\sqrt{n}(\hat{\beta}_{n}^{qmle}-\hat{\beta}_{n}^{n})\xrightarrow{P} 0. \nonumber
\end{align}

\medskip 

This paper is organized as follows. In Section \ref{sec:objective}, we propose an update formula and give a main theorem with the proof. 
We consider an example and present the simulation result to illustrate the theory in Section \ref{sec:sim}.

\section{Main result}
\label{sec:objective}
We consider the scaled Wiener process with drift \eqref{ex:wp-2},
where $\mu$ is a random and time-varying nuisance
process and $\beta\in \Theta_{\beta} \subset \mathbb{R}_{+}$. 
We denote by $(P_{\beta}: \beta\in\Theta_{\beta})$ the family of distributions of $X$ associated with $\beta$: 
$P_{\beta}(B):=P\circ X^{-1}(B)$, $B\in\mcb\big(\mcc(\mbbrp;\mbbr)\big)$. 
Let $\mcf_{t}:=\sig\{X_{0}, (w_{s})_{s\le t}\}$, the underlying filtration. 
Suppose that we observe only discrete-time sample $\mbX_{n}:=(X_{t_{j}^{n}})_{j=0}^{n}$ for $t_{j}^{n}=jh_{n}$ with $h_{n}\to 0$, 
where $T_{n}=nh_{n}\to\infty$. 
We wish to estimate the true value $\beta_{0}\in \Theta_{\beta}$ based on $\mbX_{n}$. 
As is well known, the quasi-likelihood function, denote by $\mathbb{M}_{n}$, with regarding $\mu\equiv 0$ is
\begin{align}
\mathbb{M}_{n}(\beta)=\sum_{j=1}^{n}\log\frac{1}{\sqrt{2\pi \beta h_{n}}}{\rm exp}\left(-\frac{(\Delta_{j}^{n}X)^{2}}{2\beta h_{n}}\right).
\label{eq:contrast}
\end{align}
Therefore, the quasi-MLE of $\beta$ is \eqref{eq:qmle-beta},
and $\hatb_{n}^{qmle}$ has the asymptotic normality
\begin{align}
\sqrt{n}(\hatb_{n}^{qmle}-\beta_{0})\xrightarrow{\mathcal{L}} N(0,2\beta_{0}^{2}) \nonumber
\end{align}
and the asymptotic efficiency. The aim of this paper is to propose an update estimation method for the diffusion parameter $\beta$ having the form
\begin{align}
\hat{\beta}_{j}^{n}=\hat{\beta}_{j-1}^{n}+F_{n,j}(\hat{\beta}_{j-1}^{n};X_{t_{j}^{n}}), \qquad j\leq n, \nonumber
\end{align}
where $F_{n,j}$ is an appropriate fine tuning function. The upper index $n$ implies that we have the data set $\mbX_{n}$. 
Under the nuisance drift element $\int \mu_{s}ds$, we ensure the asymptotic equivalence of $\hat{\beta}_{n}^{n}$ to $\hat{\beta}_{n}^{qmle}$.
Therefore, the asymptotic normality and the asymptotic efficiency of $\hat{\beta}_{n}^{n}$ hold. 

\medskip

We now define some notations: denote $a_{n}$ and $V(\beta_{0})$ by the convergence rate and the asymptotic variance of $\hat{\beta}_{n}^{qmle}$, respectively, i.e., $a_{n}=\sqrt{n}$ and $V(\beta_{0})=2\beta_{0}^{2}$.
We also define the quasi-likelihood function \eqref{eq:contrast} as $\mathbb{M}_{n}(\beta)=:\sum_{j=1}^{n}m_{n,j}(\beta)$.
Then, we propose the following update formula:
\begin{align}
\hatb_{j}^{n} &= \hatb_{j-1}^{n} + \frac{V(\hatb_{j-1}^{n})}{a_{j}^{2}}\p m_{n,j}(\hatb_{j-1}^{n}), \qquad j\leq n.
\label{rf_2}
\end{align}
Note that, in finite sample, this formula is more stable than the one which is derived by direct using of the Newton-Raphson method.
We have
\begin{align}
\frac{V(\beta)}{a_{j}^{2}}=\frac{2\beta^{2}}{j}, \qquad
\p m_{n,j}(\beta)=-\frac{1}{2\beta}+\frac{(\D_{j}^{n}X)^{2}}{2h_{n}\beta^{2}}, \qquad j\leq n, \nonumber
\end{align}
hence the update formula \eqref{rf_2} is
\begin{align}
\hat{\beta}_{j}^{n}&=\hat{\beta}_{j-1}^{n}+\frac{2(\hat{\beta}_{j-1}^{n})^{2}}{j}\left\{-\frac{1}{2\hat{\beta}_{j-1}^{n}}+\frac{(\D_{j}^{n}X)^{2}}{2h_{n}(\hat{\beta}_{j-1}^{n})^{2}}\right\} \nn \\
&=\Big(1-\frac{1}{j}\Big)\hat{\beta}_{j-1}^{n}+\frac{(\Delta_{j}^{n}X)^{2}}{jh_{n}}, \qquad j\leq n.
\label{rf-diffusion}
\end{align}
Note that right-hand side of \eqref{rf-diffusion} is positive a.s. for any initial value $\hat{\beta}_{0}^{n}\in \Theta_{\beta}$.
Then, for any $n\in \mathbb{N}$ and $j\leq n$ we define 
\begin{align}
\overline{\mu}_{n,j}:=\frac{1}{h_{n}}\int_{t_{j-1}}^{t_{j}}\mu_{s}ds, \qquad 
Y_{n,j}:=\frac{\Delta_{j}^{n}X-h_{n}\overline{\mu}_{n,j}}{\sqrt{h_{n}\beta_{0}}}. \nonumber
\end{align}
Note that $Y_{n,j}\stackrel{i.i.d.}{\sim}N(0,1)$ for $j\leq n$.
Now, we set two assumptions: 
\begin{assum}
There exists $0<\epsilon_{0}<1/2$ such that $T_{n}\lesssim n^{\epsilon_{0}}$.
\label{ass:hn-rate2}
\end{assum}
\begin{assum}
\begin{align}
\sqrt{\frac{h_{n}}{n}}\sum_{j=1}^{n}E_{\beta_{0}}\left[Y_{n,j}\overline{\mu}_{n,j}\big|\mathcal{F}_{t_{j-1}^{n}}\right]\xrightarrow{P} 0;
\label{ass:m1}
\end{align}
\begin{align}
\sup_{n>0}\sup_{t\leq T_{n}}E_{\beta_{0}}\left[|\mu_{t}|^{2}\right]<\infty.
\label{ass:m2}
\end{align}
\label{ass:moment}
\end{assum}
The notation $A_{n}\lesssim B_{n}$ in Assumption \ref{ass:hn-rate2} means that $\sup_{n}(A_{n}/B_{n})<\infty$. The following Theorem \ref{thm:mainthm} is a main result in this paper. 


\begin{thm}
Consider the model \eqref{ex:wp-2}. Assume that Assumptions \ref{ass:hn-rate2} and \ref{ass:moment} hold.
Then, for any $\hat{\beta}_{0}^{n}\in \Theta_{\beta}$ the update formula \eqref{rf-diffusion} generates an estimator $\hat{\beta}_{n}^{n}$, which is asymptotic equivalent to $\hat{\beta}_{n}^{qmle}$:
\begin{align}
\sqrt{n}(\hat{\beta}_{n}^{qmle}-\hat{\beta}_{n}^{n})\xrightarrow{P} 0.
\label{claim}
\end{align}
Therefore, it has the asymptotic normality:
\begin{align}
\sqrt{n}(\hat{\beta}_{n}^{n}-\beta_{0})\xrightarrow{\mathcal{L}}N(0,2\beta_{0}^{2}) \nonumber
\end{align}
and the asymptotic efficiency.
\label{thm:mainthm}
\end{thm}

\begin{proof}
To prove the theorem, we use the result of Sharia \cite[Lemma 1]{Sha10_aism}. We change some notations from the original ones of Sharia \cite{Sha10_aism} in terms of a triangular array of random variables. 
We drop the index $n$ to simplify some notations. 
We can rewrite \eqref{rf-diffusion} to
\begin{align}
\hat{\beta}_{j}&=\hat{\beta}_{j-1}+\frac{2\hat{\beta}_{j-1}^{2}}{j}\left\{-\frac{1}{2\hat{\beta}_{j-1}}+\frac{(\sqrt{h_{n}\beta_{0}}Y_{j}+h_{n}\overline{\mu}_{j})^{2}}{2h_{n}\hat{\beta}_{j-1}^{2}}\right\} \nn \\
&=\hat{\beta}_{j-1}+\frac{2\hat{\beta}_{j-1}^{2}}{j}\left\{-\frac{1}{2\hat{\beta}_{j-1}}+\frac{\beta_{0}}{2\hat{\beta}_{j-1}^{2}}\left(Y_{j}+\overline{\mu}_{j}\sqrt{\frac{h_{n}}{\beta_{0}}}\right)^{2}\right\} \nn \\
&=:\hat{\beta}_{j-1}+\Gamma_{j}(\hat{\beta}_{j-1})^{-1}\cdot \psi_{n,j}(\hat{\beta}_{j-1}), \nonumber 
\end{align}
where 
\begin{align}
\Gamma_{j}(\beta):=
\frac{j}{2\beta^{2}}, \qquad
\psi_{n,j}(\beta):=
-\frac{1}{2\beta}+\frac{\beta_{0}}{2\beta^{2}}\left(Y_{j}+\overline{\mu}_{j}\sqrt{\frac{h_{n}}{\beta_{0}}}\right)^{2}. \nonumber
\end{align}
Moreover, we define the following notations:
\begin{align}
b_{n,j}(\beta_{0},d_{j-1})=E_{\beta_{0}}\left[\psi_{n,j}(\hat{\beta}_{j-1})\big|\mathcal{F}_{t_{j-1}}\right]; \nonumber
\end{align}
\begin{align}
R_{n,j}(\beta_{0},d_{j-1})=\Gamma_{j}(\beta_{0})\Gamma_{j}^{-1}(\hat{\beta}_{j-1})b_{n,j}(\beta_{0},d_{j-1}); \nonumber
\end{align}
\begin{align}
\mathcal{E}_{n,j}(\beta_{0},d_{j-1})=\Gamma_{j}(\beta_{0})\Gamma_{j}^{-1}(\hat{\beta}_{j-1})\{\psi_{n,j}(\hat{\beta}_{j-1})-b_{n,j}(\beta_{0},d_{j-1})\}-\psi_{n,j}(\beta_{0}), \nonumber
\end{align}
where $d_{j}:=\hat{\beta}_{j}-\beta_{0}$ and $j \leq n$, and introduce the conditions of Sharia \cite[Lemma 1]{Sha10_aism} applied to our model setting:
\begin{enumerate}
\item
\begin{align}
\frac{(\sqrt{n})^{2}}{\Gamma_{n}(\beta_{0})}\xrightarrow{\mathcal{L}} \eta(\beta_{0}) \nn
\end{align}
w.r.t. $P_{\beta_{0}}$, where $\eta(\beta_{0})$ is a random variable with $\eta(\beta_{0})<\infty\ P_{\beta_{0}}{\rm \mathchar`-a.s.}$;
\item
\begin{align}
\lim_{n\rightarrow \infty}\frac{1}{\sqrt{n}}\sum_{j=1}^{n}\big\{\big(\Gamma_{j}(\beta_{0})-\Gamma_{j-1}(\beta_{0})\big)d_{j-1}+R_{n,j}(\beta_{0},d_{j-1})\big\}=0 \nn
\end{align}
in probability $P_{\beta_{0}}$;
\item
\begin{align}
\lim_{n\rightarrow \infty}\frac{1}{\sqrt{n}}\sum_{j=1}^{n}\mathcal{E}_{n,j}(\beta_{0},d_{j-1})=0 \nn
\end{align}
in probability $P_{\beta_{0}}$.
\end{enumerate}

\bigskip

Now, let us show the above conditions and derive the {\it local asymptotic linearity} of $\hat{\beta}_{n}^{n}$:
\begin{align}
\sqrt{n}(\hat{\beta}_{n}^{n}-\hat{\beta}_{n}^{\ast})\xrightarrow{P} 0,
\label{eq:L.A.L-beta}
\end{align}
where
\begin{align}
\hat{\beta}_{n}^{\ast}=\beta_{0}+\Gamma_{n}^{-1}(\beta_{0})\sum_{j=1}^{n}\psi_{n,j}(\beta_{0}) \nonumber
\end{align}
is a linear statistic. Clearly, (i) holds since we have $n/\Gamma_{n}(\beta_{0})=2\beta_{0}^{2}$. To check (ii), we calculate 
\begin{align}
b_{n,j}(\beta_{0},d_{j-1})
&=-\frac{1}{2\hat{\beta}_{j-1}}+\frac{\beta_{0}}{2\hat{\beta}_{j-1}^{2}}E_{\beta_{0}}\left[\left(Y_{j}+\overline{\mu}_{j}\sqrt{\frac{h_{n}}{\beta_{0}}}\right)^{2}\Big|\mathcal{F}_{t_{j-1}}\right] \nn \\
&=-\frac{1}{2\hat{\beta}_{j-1}}+\frac{\beta_{0}}{2\hat{\beta}_{j-1}^{2}}\left(1+2\sqrt{\frac{h_{n}}{\beta_{0}}}E_{\beta_{0}}\left[Y_{j}\overline{\mu}_{j}\big|\mathcal{F}_{t_{j-1}}\right]+\frac{h_{n}}{\beta_{0}}E_{\beta_{0}}\left[|\overline{\mu}_{j}|^{2}\big|\mathcal{F}_{t_{j-1}}\right]\right) \nn \\
&=-\frac{d_{j-1}}{2\hat{\beta}_{j-1}^{2}}+\frac{\sqrt{h_{n}\beta_{0}}}{\hat{\beta}_{j-1}^{2}}E_{\beta_{0}}\left[Y_{j}\overline{\mu}_{j}\big|\mathcal{F}_{t_{j-1}}\right]+\frac{h_{n}}{2\hat{\beta}_{j-1}^{2}}E_{\beta_{0}}\left[|\overline{\mu}_{j}|^{2}\big|\mathcal{F}_{t_{j-1}}\right], \nonumber
\end{align}
\begin{align}
R_{n,j}(\beta_{0},d_{j-1})
&=\frac{j}{2\beta_{0}^{2}}\frac{2\hat{\beta}_{j-1}^{2}}{j}\left\{-\frac{d_{j-1}}{2\hat{\beta}_{j-1}^{2}}+\frac{\sqrt{h_{n}\beta_{0}}}{\hat{\beta}_{j-1}^{2}}E_{\beta_{0}}\left[Y_{j}\overline{\mu}_{j}\big|\mathcal{F}_{t_{j-1}}\right]+\frac{h_{n}}{2\hat{\beta}_{j-1}^{2}}E_{\beta_{0}}\left[|\overline{\mu}_{j}|^{2}\big|\mathcal{F}_{t_{j-1}}\right]\right\} \nn \\
&=\frac{1}{\beta_{0}^{2}}\left(-\frac{d_{j-1}}{2}+\sqrt{h_{n}\beta_{0}}E_{\beta_{0}}\left[Y_{j}\overline{\mu}_{j}\big|\mathcal{F}_{t_{j-1}}\right]+\frac{h_{n}}{2}E_{\beta_{0}}\left[|\overline{\mu}_{j}|^{2}\big|\mathcal{F}_{t_{j-1}}\right]\right). \nonumber
\end{align}
Hence, we have 
\begin{align}
&\big(\Gamma_{j}(\beta_{0})-\Gamma_{j-1}(\beta_{0})\big)d_{j-1}+R_{n,j}(\beta_{0},d_{j-1}) \nn \\
&=\left(\frac{j}{2\beta_{0}^{2}}-\frac{j-1}{2\beta_{0}^{2}}\right)d_{j-1}+\frac{1}{\beta_{0}^{2}}\left(-\frac{d_{j-1}}{2}+\sqrt{h_{n}\beta_{0}}E_{\beta_{0}}\left[Y_{j}\overline{\mu}_{j}\big|\mathcal{F}_{t_{j-1}}\right]+\frac{h_{n}}{2}E_{\beta_{0}}\left[|\overline{\mu}_{j}|^{2}\big|\mathcal{F}_{t_{j-1}}\right]\right) \nn \\
&=\frac{1}{\beta_{0}^{2}}\left(\sqrt{h_{n}\beta_{0}}E_{\beta_{0}}\left[Y_{j}\overline{\mu}_{j}\big|\mathcal{F}_{t_{j-1}}\right]+\frac{h_{n}}{2}E_{\beta_{0}}\left[|\overline{\mu}_{j}|^{2}\big|\mathcal{F}_{t_{j-1}}\right]\right), \nonumber
\end{align}
and therefore
\begin{align}
&\frac{1}{\sqrt{n}}\sum_{j=1}^{n}\left\{(\Gamma_{j}(\beta_{0})-\Gamma_{j-1}(\beta_{0}))d_{j-1}+R_{n,j}(\beta_{0},d_{j-1})\right\} \nn \\
&\lesssim \sqrt{\frac{h_{n}}{n}}\sum_{j=1}^{n}E_{\beta_{0}}\left[Y_{j}\overline{\mu}_{j}\big|\mathcal{F}_{t_{j-1}}\right]+\frac{h_{n}}{\sqrt{n}}\sum_{j=1}^{n}E_{\beta_{0}}\left[|\overline{\mu}_{j}|^{2}\big|\mathcal{F}_{t_{j-1}}\right].
\label{eq:(ii)}
\end{align}
The first term of \eqref{eq:(ii)} converges to $0$ in probability $P_{\beta_{0}}$ since we assume \eqref{ass:m1}. Let us show
\begin{align}
Z_{n}:=\frac{h_{n}}{\sqrt{n}}\sum_{j=1}^{n}E_{\beta_{0}}\left[|\overline{\mu}_{j}|^{2}\big|\mathcal{F}_{t_{j-1}}\right]\rightarrow 0 \nonumber
\end{align}
in probability $P_{\beta_{0}}$. Then, by making use of Markov's inequality, Assumption \ref{ass:hn-rate2} and \eqref{ass:m2}, we obtain for any $\epsilon>0$
\begin{align}
P_{\beta_{0}}\left(|Z_{n}|>\epsilon\right)\leq \frac{h_{n}}{\epsilon\sqrt{n}}\sum_{j=1}^{n}E_{\beta_{0}}\left[|\overline{\mu}_{j}|^{2}\right]\lesssim h_{n}\sqrt{n}\rightarrow 0, \nonumber
\end{align}
hence (ii). To check the condition (iii), we also calculate
\begin{align}
&\mathcal{E}_{n,j}(\beta_{0},d_{j-1}) \nonumber \\ 
&=\frac{j}{2\beta_{0}^{2}}\frac{2\hat{\beta}_{j-1}^{2}}{j}\left\{-\frac{1}{2\hat{\beta}_{j-1}}+\frac{\beta_{0}}{2\hat{\beta}_{j-1}^{2}}\left(Y_{j}+\overline{\mu}_{j}\sqrt{\frac{h_{n}}{\beta_{0}}}\right)^{2}+\frac{d_{j-1}}{2\hat{\beta}_{j-1}^{2}} \right.\nn \\
&\ \ \ \left.-\frac{\sqrt{h_{n}\beta_{0}}}{\hat{\beta}_{j-1}^{2}}E_{\beta_{0}}\left[Y_{j}\overline{\mu}_{j}\big|\mathcal{F}_{t_{j-1}}\right]-\frac{h_{n}}{2\hat{\beta}_{j-1}^{2}}E_{\beta_{0}}\left[\overline{\mu}_{j}^{2}\big|\mathcal{F}_{t_{j-1}}\right]\right\}-\left\{-\frac{1}{2\beta_{0}}+\frac{\beta_{0}}{2\beta_{0}^{2}}\left(Y_{j}+\overline{\mu}_{j}\sqrt{\frac{h_{n}}{\beta_{0}}}\right)^{2}\right\} \nn \\
&=\frac{1}{\beta_{0}^{2}}\left\{-\frac{\hat{\beta}_{j-1}}{2}+\frac{\beta_{0}}{2}\left(Y_{j}+\overline{\mu}_{j}\sqrt{\frac{h_{n}}{\beta_{0}}}\right)^{2}+\frac{d_{j-1}}{2}-\sqrt{h_{n}\beta_{0}}E_{\beta_{0}}\left[Y_{j}\overline{\mu}_{j}\big|\mathcal{F}_{t_{j-1}}\right]-\frac{h_{n}}{2}E_{\beta_{0}}\left[|\overline{\mu}_{j}|^{2}\big|\mathcal{F}_{t_{j-1}}\right]\right\} \nn \\
&\ \ \ +\frac{1}{2\beta_{0}}-\frac{1}{2\beta_{0}}\left(Y_{j}+\overline{\mu}_{j}\sqrt{\frac{h_{n}}{\beta_{0}}}\right)^{2} \nn \\
&=\frac{1}{\beta_{0}^{2}}\left(-\sqrt{h_{n}\beta_{0}}E_{\beta_{0}}\left[Y_{j}\overline{\mu}_{j}\big|\mathcal{F}_{t_{j-1}}\right]-\frac{h_{n}}{2}E_{\beta_{0}}\left[|\overline{\mu}_{j}|^{2}\big|\mathcal{F}_{t_{j-1}}\right]\right). \nonumber
\end{align}
Assumptions \ref{ass:hn-rate2} and \ref{ass:moment} also ensure the condition (iii):
\begin{align}
\lim_{n\rightarrow\infty}\frac{1}{\sqrt{n}}\sum_{j=1}^{n}\mathcal{E}_{j}(\beta_{0},d_{j-1})=0 \nonumber
\end{align}
in probability $P_{\beta_{0}}$. Consequently, we derive the local asymptotic linearity \eqref{eq:L.A.L-beta} from Sharia \cite[Lemma 1]{Sha10_aism}.
In this case, we obtain
\begin{align}
\hat{\beta}_{n}^{\ast}&=\beta_{0}+\frac{2\beta_{0}^{2}}{n}\sum_{j=1}^{n}\left\{-\frac{1}{2\beta_{0}}+\frac{\beta_{0}}{2\beta_{0}^{2}}\left(Y_{j}+\overline{\mu}_{j}\sqrt{\frac{h_{n}}{\beta_{0}}}\right)^{2}\right\} \nonumber \\
&=\frac{1}{T_{n}}\sum_{j=1}^{n}(\Delta_{j}^{n}X)^{2}=\hat{\beta}_{n}^{qmle}, \nonumber 
\end{align}
therefore we can conclude the claim \eqref{claim}.
\end{proof}

\begin{rem}
\upshape
See e.g., Jacod and Shiryaev \cite{JacShi87} for more detailed discussion on the asymptotic properties of an asymptotically linear estimator. \hfill $\Box$
\end{rem}

\section{Simulation}
\label{sec:sim}
We performed a simulation study to validate our proposed update formula. 
We consider the model
\begin{align}
X_{t}=\al t + \sqrt{\beta}w_{t}, \nonumber
\label{ex:wp-1}
\end{align}
where we know the true value $\alpha_{0} \in \mathbb{R}$. Note that this means $\mu_{s}\equiv \al_{0}$, hence Assumption \ref{ass:moment} holds. 
We set $(\alpha_{0},\beta_{0})=(-1,3)$, data size $N=500$ and the Monte Carlo trial number $L=100$. We also set $h_{N}=N^{-3/4}$ to satisfy Assumption \ref{ass:hn-rate2}. We generate $\{\hat{\beta}_{j}^{N}\}_{j=0}^{N}$ through the following steps:
\begin{itemize}
\item Step 1: Generate a initial value $\hat{\beta}_{0}^{N}$ by $\beta_{0}+\beta_{0}U$, where $U\sim U(-0.5,0.5)$.
\item Step 2: Update $\hat{\beta}_{j}^{N}$ from $j=1$ to $j=N$ by \eqref{rf-diffusion} with $\Delta_{j}^{N}X\stackrel{i.i.d.}{\sim}N(\alpha_{0}h_{N},\beta_{0}h_{N})$, and obtain $\{\hat{\beta}_{j}^{N}\}_{j=0}^{N}$.
\end{itemize}
We repeat Step 1 and 2 $L$ times. The following Figures preset the simulation results. Figures \ref{fig:sam} and \ref{fig:beta-trace} are traces of data $\mbX_{N}=(X_{t_{j}^{N}})_{j=1}^{N}$ and $\hat{\beta}_{j}^{N}$ at $l=100$, respectively ($l(\leq L)$ denote each steps of Monte Carlo trials). $y$-axis has update times $j\leq N$. The dotted line in Figure \ref{fig:beta-trace} denotes the true value $\beta_{0}=3$. In Figure \ref{fig:average-trace} 
we plot $(j,\overline{\hat{\beta}_{j}^{N}})$ for $j\leq N$, where
\begin{align}
\overline{\hat{\beta}_{j}^{N}}:=\frac{1}{L}\sum_{l=1}^{L}\hat{\beta}_{j}^{N,l}, \nonumber
\end{align}
and $\hat{\beta}_{j}^{N,l}$ denotes the value of $\hat{\beta}_{j}^{N}$ at $l\leq L$. We give a QQ-plot for $W_{N}:=\sqrt{N/2\beta_{0}^{2}}(\hat{\beta}_{N}^{N}-\beta_{0})$ in Figure \ref{fig:qqplot}.
From the Figure \ref{fig:average-trace}, 
we expect that $\hat{\beta}_{N}^{N}\xrightarrow{P}\beta_{0}$ as $N\rightarrow \infty$. Actually, whenever we repeat this algorithm, $\overline{\hat{\beta}_{N}^{N}}$
nearly equal to $\beta_{0}$ and standard deviations are not so large (about $0.19$). We also expect that
\begin{align}
W_{N}\xrightarrow{\mathcal{L}}N(0,1) \nonumber
\end{align}
as $N\rightarrow \infty$ from the QQ-plot in Figure \ref{fig:qqplot} since the plots almost lay on the 45 degree line.

\begin{figure}[htbp]
\begin{minipage}{0.49\hsize}
\begin{center}
\includegraphics[width=70mm, height=70mm]{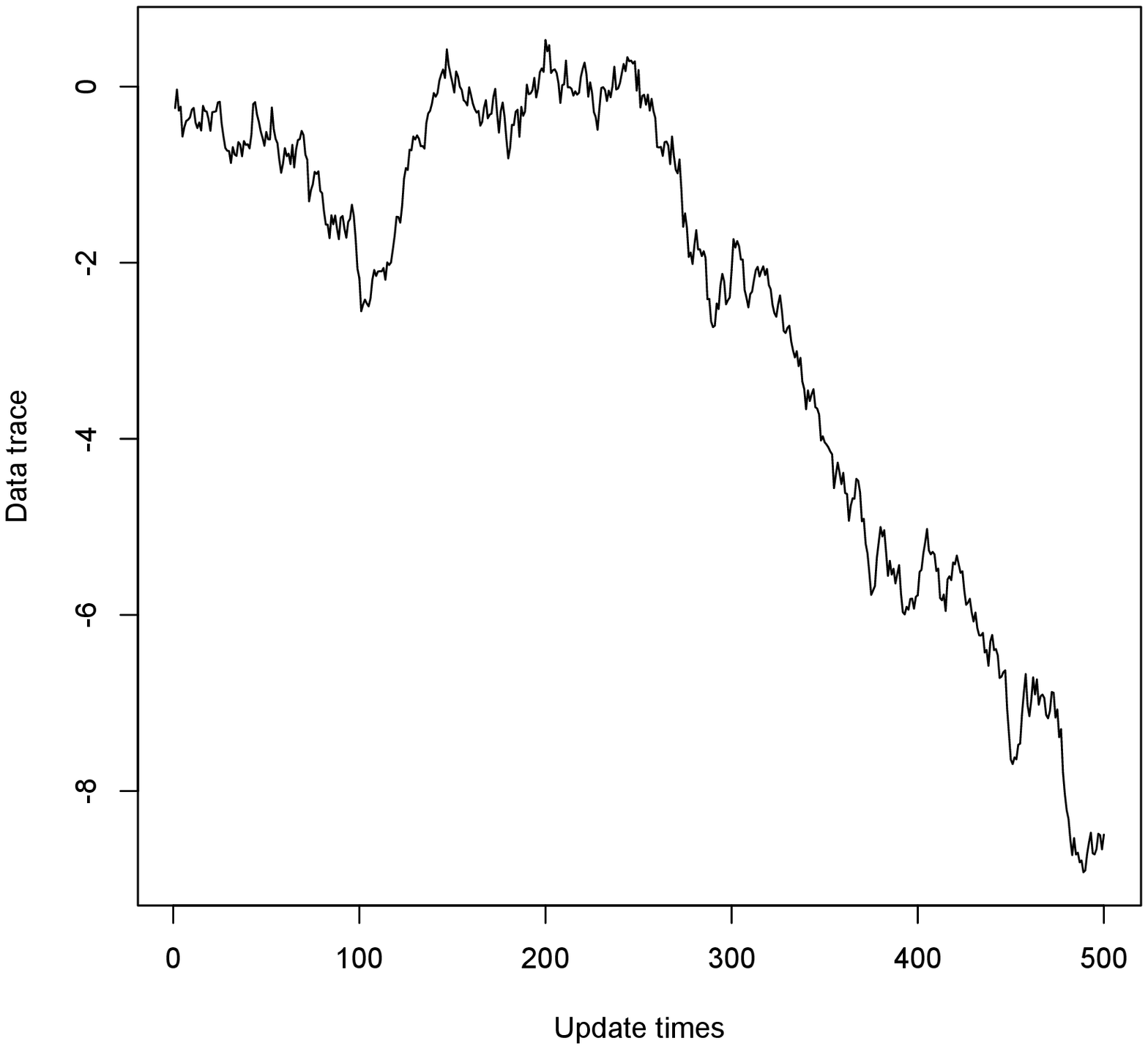}
\end{center}
\vspace{-5mm}
\caption{Trace of data $\mbX_{N}$.}
\label{fig:sam}
\end{minipage}
\begin{minipage}{0.49\hsize}
\begin{center}
\includegraphics[width=70mm, height=70mm]{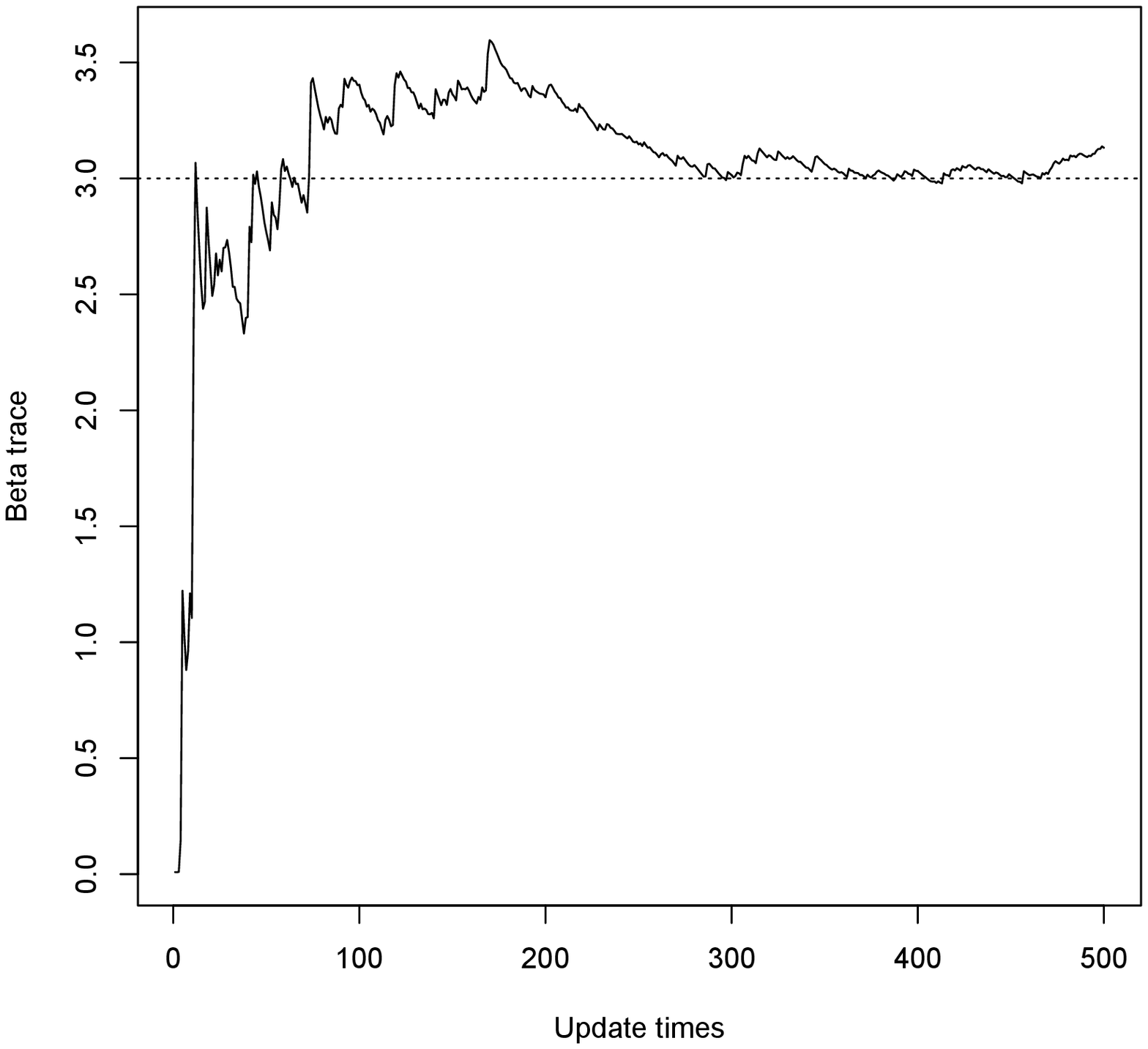}
\end{center}
\vspace{-5mm}
\caption{Trace of $\hat{\beta}_{j}^{N}$.}
\label{fig:beta-trace}
\end{minipage}
\begin{minipage}{0.49\hsize}
\begin{center}
\includegraphics[width=70mm, height=70mm]{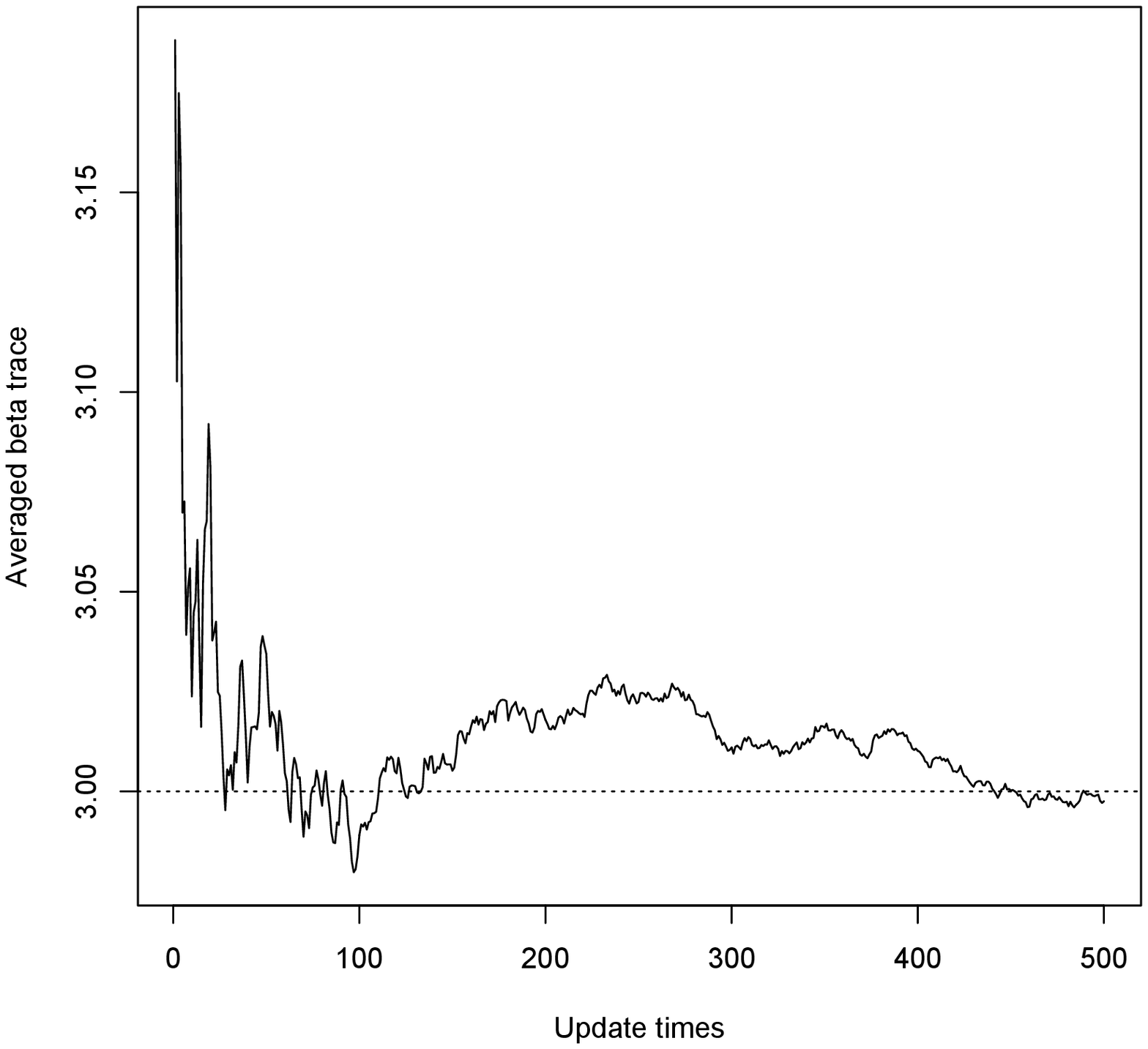}
\end{center}
\vspace{-5mm}
\caption{Trace of $\overline{\hat{\beta}_{j}^{N}}$.}
\label{fig:average-trace}
\end{minipage}
\begin{minipage}{0.49\hsize}
\begin{center}
\includegraphics[width=70mm, height=70mm]{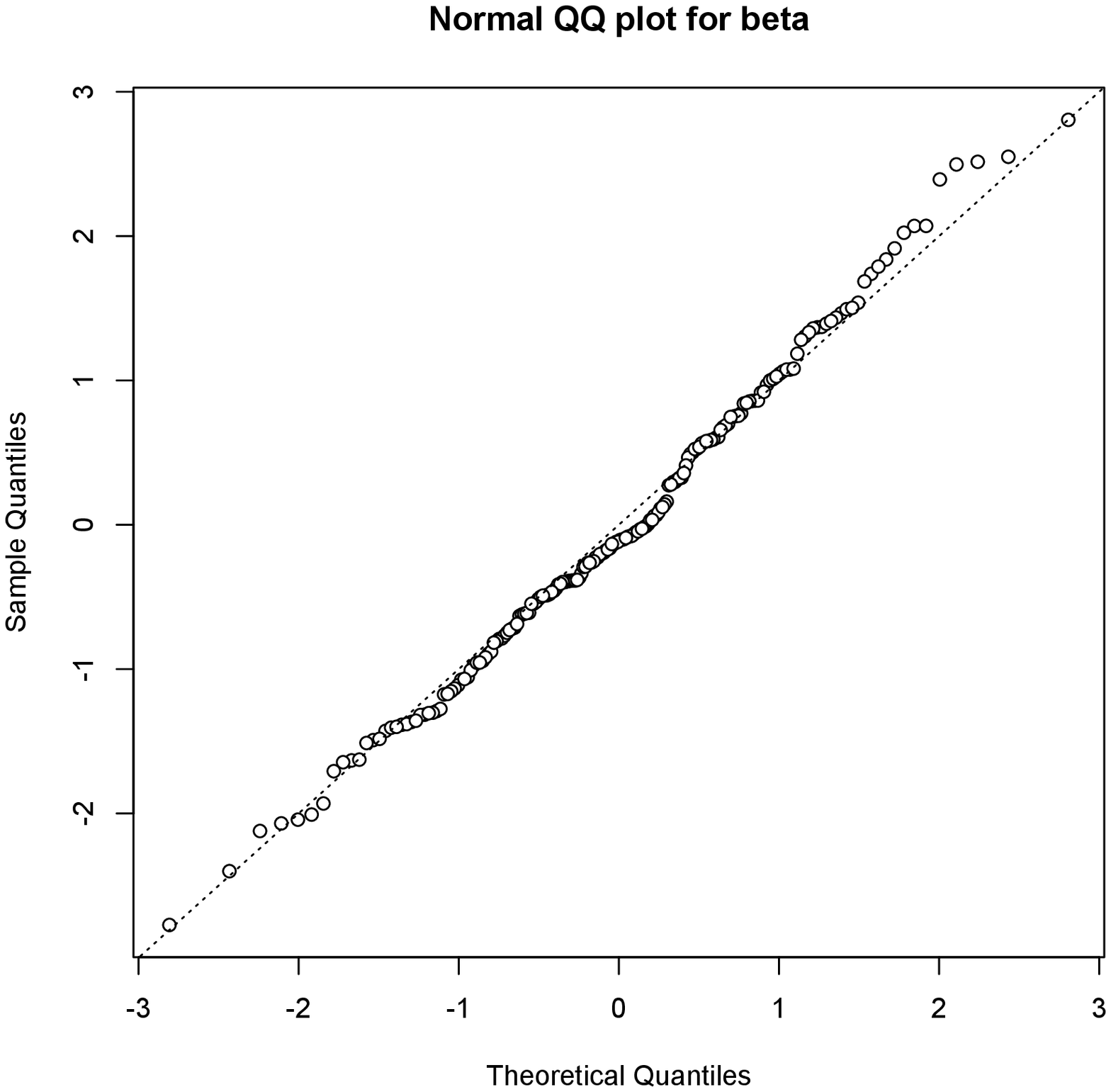}
\end{center}
\vspace{-5mm}
\caption{{\small QQ-plot for $W_{N}$.}}
\label{fig:qqplot}
\end{minipage}
\end{figure}


\begin{ack}
\upshape
The author is grateful to Professor H. Masuda, Kyushu university, for his valuable comments.
\end{ack}


\def\cprime{$'$} \def\polhk#1{\setbox0=\hbox{#1}{\ooalign{\hidewidth
  \lower1.5ex\hbox{`}\hidewidth\crcr\unhbox0}}} \def\cprime{$'$}
  \def\cprime{$'$}

\end{document}